\documentclass[12pt,a4paper,oneside,reqno]{amsart}

\usepackage{amsthm, amsmath, amssymb,amscd}
\usepackage[T1]{fontenc}
\usepackage[utf8]{inputenc}
\usepackage[english]{babel}
\usepackage{typearea}
\usepackage{yfonts}
\usepackage{mathrsfs}
\usepackage{hyperref}
\usepackage{fixme}
\usepackage{verbatim}
\usepackage{color}

\usepackage{tikz-cd}	
\usetikzlibrary{matrix,arrows,decorations.markings}

\newcommand{\N}{\mathbb{N}}     
\newcommand{\Z}{\mathbb{Z}}     
\newcommand{\C}{\mathbb{C}}     
\newcommand{\T}{\mathbb{T}}     
\newcommand{\K}{\mathbb{K}}     
\newcommand{\F}{\mathcal{F}}    
\newcommand{\X}{\mathsf{X}}     
\newcommand{\G}{\mathcal{G}}    
\newcommand{\OO}{\mathcal{O}}   
\newcommand{\D}{\mathcal{D}}    
\newcommand{\LRA}{\longrightarrow}

{}
{}             
\let\geq\geqslant{}
\let\subset\subseteq{}          
{}

\newtheorem{lemma}{Lemma}[section]      
\newtheorem{corollary}[lemma]{Corollary}
\newtheorem{theorem}[lemma]{Theorem}

\theoremstyle{definition}

\newtheorem{example}[lemma]{Example}

\newtheorem{remark}[lemma]{Remark}

\begin{document}

\begin{abstract}
A one-sided shift of finite type $(\X_A,\sigma_A)$ determines on the one hand a Cuntz--Krieger algebra $\OO_A$ with a distinguished abelian subalgebra $\D_A$ and a certain completely positive map $\tau_A$ on $\OO_A$. On the other hand, $(\X_A,\sigma_A)$ determines a groupoid $\G_A$ together with a certain homomorphism $\epsilon_A$ on $\G_A$. We show that this completely characterizes the one-sided conjugacy class of $\X_A$. This strengthens a result of J.~Cuntz and W.~Krieger. We also exhibit an example of two irreducible shifts of finite type which are eventually conjugate but not conjugate. 
This provides a negative answer to a question of K.~Matsumoto of whether eventual conjugacy implies conjugacy.
\end{abstract}

\title[Cuntz--Krieger algebras and one-sided conjugacy of shifts of finite type]{Cuntz--Krieger algebras and one-sided conjugacy of shifts of finite type and their groupoids}

\author{Kevin Aguyar Brix}
\address[K.A. Brix]{Department of Mathematical Sciences\\University of Copenhagen\\Universitetsparken 5\\DK-2100 Copenhagen\\Denmark}
\email{kab@math.ku.dk}

\author{Toke Meier Carlsen}
\address[T.M. Carlsen]{Department of Science and Technology\\University of the Faroe Islands\\
N\'oat\'un 3\\ FO-100 T\'orshavn\\the Faroe Islands}
\email{toke.carlsen@gmail.com}

\keywords{Shifts of finite type, groupoids, Cuntz--Krieger algebras}
\subjclass[2010]{primary 46L55, secondary 37A55, 37B10}
\thanks{The first named author is supported by the Danish National Research Foundation through the Centre for Symmetry and Deformation (DNRF92)}

\maketitle


\section{Introduction}

In~\cite{CK80}, J.~Cuntz and W.~Krieger initiated what has turned out to be a very fascinating study of the symbiotic relationship between operator algebras and symbolic dynamics. 
Given a finite square $\{0,1\}$-matrix $A$ with no zero rows or zero columns, 
they construct a $C^*$-algebra $\OO_A$ which is now called the \emph{Cuntz--Krieger algebra} of $A$ with a distinguished abelian $C^*$-subalgebra $\D_A$ called the \emph{diagonal} subalgebra.
Under a certain condition (I) (this is later generalized to condition (L) of graphs), 
$\OO_A$ is a universal $C^*$-algebra and there is an action of the circle group $\T\curvearrowright \OO_A$ called the \emph{gauge action}.

The matrix $A$ also determines both a one-sided and a two-sided shift space of finite type (see, e.g.,~\cite{Kitchens},~\cite{LM}) denoted $(\X_A,\sigma_A)$ and $(\bar{\X}_A,\bar{\sigma}_A)$, respectively. 
In fact, any one-sided (resp., two-sided) shift of finite type is conjugate to $(\X_A,\sigma_A)$ (resp., $(\bar{\X}_A,\bar{\sigma}_A)$) for some finite square $\{0,1\}$-matrix $A$ with no zero rows or zero columns. The spectrum of the above mentioned abelian $C^*$-subalgebra $\D_A$ is homeomorphic to $\X_A$ in a natural way. 

In~\cite[Proposition 2.17]{CK80}, J.~Cuntz and W.~Krieger proved that the data consisting of the $C^*$-algebra $\OO_A$, the diagonal $\D_A$, the gauge action and the restriction to the diagonal of a certain completely positive map $\phi_A\colon \OO_A\LRA \OO_A$ is an invariant of the one-sided conjugacy class of $\X_A$ provided $A$ satisfies condition (I).

In~\cite{Cuntz81}, J.~Cuntz also showed the following two results under the hypothesis that the defining matrices as well as their transposes satisfy condition (I):
A flow equivalence between the two-sided shift spaces $(\bar{\X}_A,\bar{\sigma}_A)$ and $(\bar{\X}_B,\bar{\sigma}_B)$ implies the existence of a $*$-isomorphism $\OO_A\otimes \K \LRA \OO_B\otimes \K$ which sends $\D_A\otimes \mathcal{C}$ onto $\D_B\otimes \mathcal{C}$,
where $\mathcal{C}$ is the canonical maximal abelian subalgebra of the $C^*$-algebra $\K$ of compact operators on an infinite-dimensional separable Hilbert space;
a two-sided conjugacy between $(\bar{\X}_A,\bar{\sigma}_A)$ and $(\bar{\X}_B,\bar{\sigma}_B)$ implies the existence of a diagonal-preserving $*$-isomorphism between the stabilized $C^*$-algebras as above which also intertwines the gauge actions.
Both of these results were also present in~\cite{CK80} under the additional assumptions that the defining matrices are irreducible and aperiodic.

From a one-sided shift space $(\X_A,\sigma_A)$ one can construct an amenable locally compact Hausdorff étale groupoid $\G_A$ whose $C^*$-algebra $C^*(\G_A)$ is isomorphic to $\OO_A$ in a way that maps $\D_A$ onto $C(\G_A^{(0)})$. 
Using this groupoid, H.~Matui and K.~Matsumoto in~\cite{MM14} improved the work of J.~Cuntz and W.~Krieger when they showed that flow equivalence between shift spaces determined by irreducible and non-permutation $\{0,1\}$-matrices $A$ and $B$ is equivalent to the existence of a diagonal-preserving $*$-isomorphism between the stabilized Cuntz--Krieger algebras.
This result was later proved in~\cite{CEOR} to hold for any pair of finite square $\{0,1\}$-matrices $A$ and $B$ with no zero rows and no zero colums. 
In~\cite{CR}, the second-named author and J.~Rout used a similar approach to prove that $(\bar{\X}_A,\bar{\sigma}_A)$ and $(\bar{\X}_B,\bar{\sigma}_B)$ are conjugate 
(only assuming that $A$ and $B$ have no zero rows or zero columns) 
if and only if there exists a diagonal preserving $*$-isomorphism between the stabilized Cuntz--Krieger algebras which intertwines the gauge actions.
Thus, one can recover the two-sided shift space $(\bar{\X}_A,\bar{\sigma}_A)$ both up to flow equivalence and up to conjugacy from its Cuntz--Krieger algebra $\OO_A$.

In~\cite{Mat10}, K.~Matsumoto introduced the notion of continuous orbit equivalence of one-sided shift spaces of finite type and showed that for irreducible and non-permutation $\{0,1\}$-matrices $A$ and $B$, the one-sided shift spaces $(\X_A,\sigma_A)$ and $(\X_B,\sigma_B)$ are continuously orbit equivalent if and only if there is a diagonal preserving $*$-isomorphism (i.e., a $*$-isomorphism that maps $\D_A$ onto $\D_B$) between the Cuntz--Krieger algebras $\OO_A$ and $\OO_B$. Building on the reconstruction theory of J.~Renault in~\cite{Renault08}, H.~Matui and K.~Matsumoto observed in~\cite{MM14} that this is equivalent to isomorphism of the groupoids $\G_A$ and $\G_B$. These results were in~\cite{CEOR} shown to hold for any pair of finite square $\{0,1\}$-matrices $A$ and $B$ with no zero rows and no zero colums. 

Given irreducible and non-permutation $\{0,1\}$-matrices $A$ and $B$, K.~Matsumoto proved that the stronger notion of eventual conjugacy of one-sided shifts is completely characterized by diagonal-preserving $*$-isomorphism of the Cuntz--Krieger algebras which intertwines the gauge actions, see~\cite{Mat17cts}. 
This is generalized to any pair of finite square $\{0,1\}$-matrices $A$ and $B$ with no zero rows and no zero colums in~\cite{CR}. 
K.~Matsumoto then asks the question whether eventual conjugacy is equivalent to conjugacy. 
In the wake of the above mentioned results, the question can be rephrased as to whether 
diagonal preserving $*$-isomorphism of Cuntz--Krieger algebras intertwining the gauge actions actually characterizes conjugacy (see~\cite[Remark 3.6]{Mat17cts} or~\cite[p. 1139]{Mat17uni}).

In this paper, we address this question and the characterization of one-sided conjugacy of finite type shift spaces in relation to~\cite[Proposition 2.17]{CK80}.
Given any finite square $\{0,1\}$-matrix $A$ with no zero rows or zero columns, we introduce a continuous groupoid homomorphism $\epsilon_A\colon \G_A\LRA\G_A$ 
which induces a completely positive map $\tau_A\colon \OO_A\LRA\OO_A$.
This is different but related to the completely positive map $\phi_A\colon \OO_A\LRA\OO_A$ considered in~\cite{CK80}.
We show that for a pair of finite square $\{0,1\}$-matrices $A$ and $B$ with no zero rows and no zero colums, 
the one-sided shift spaces $(\X_A,\sigma_A)$ and $(\X_B,\sigma_B)$ are conjugate if and only if there is a diagonal preserving $*$-isomorphism 
$\Psi\colon \OO_A\LRA\OO_B$ that intertwines $\tau_A$ and $\tau_B$, 
if and only if there is a groupoid isomorphism $\Phi\colon \G_A\LRA\G_B$ that intertwines $\epsilon_A$ and $\epsilon_B$. 
We also show that these conditions are equivalent to the existence of a diagonal preserving $*$-isomorphism between $\OO_A$ and $\OO_B$ 
that intertwines both the gauge actions and $\phi_A|_{\D_A}$ and $\phi_B|_{\D_B}$, and thus show that~\cite[Proposition 2.17]{CK80} holds also for matrices that do not satisfy condition (I). 
Finally, we exhibit an example of two irreducible shifts of finite type which are eventually conjugate but not conjugate.
This shows that conjugacy is strictly stronger than eventual conjugacy and this answers K.~Matsumoto's question in the negative.

\section{Notation and preliminaries}

In this section we briefly recall the definitions of the one-sided shift space $\X_A$, the groupoid $\G_A$, and the Cuntz--Krieger algebra $\OO_A$ together with the subalgebra $\D_A$ and the gauge action $\gamma\colon \T\curvearrowright \OO_A$.
We let $\Z$, $\N=\{0,1,2,\ldots\}$ and $\C$ denote the integers, the non-negative integers and the complex numbers, respectively.
Let $\T\subset \C$ be the unit circle group.

\subsection{One-sided shifts of finite type}
Let $N$ be a positive integer and let $A\in M_N(\{0,1\})$ be a matrix with no zero rows and no zero columns. The set
\[
    \X_A := \{ x = {(x_n)}_{n\in \N} \in {\{1,\ldots,N\}}^{\N} \mid A(x_n,x_{n+1}) = 1,~n\in \N\}
\]
is a second countable compact Hausdorff space in the subspace topology of $\{1,\ldots,N\}^{\N}$ (equipped with the product topology). 
Together with the shift operation $\sigma_A\colon \X_A\LRA \X_A$ given by ${(\sigma_A(x))}_n = x_{n+1}$ for $x\in \X_A$ and $n\in \N$, 
the pair $(\X_A,\sigma_A)$ is \emph{the one-sided shift of finite type} determined by $A$. 
In the literature (e.g.,~\cite{Cuntz81},~\cite{CK80},~\cite{Mat10},~\cite{MM14}), $(\X_A,\sigma_A)$ is often refered to as the \emph{one-sided topological Markov chain} determined by $A$. 
The reader is refered to~\cite{LM} for an excellent introduction to the general theory of shift spaces.

A finite string $\alpha = \alpha_0 \alpha_1\cdots \alpha_{n - 1}$ with $\alpha_i\in \{1,\ldots,N\}$ is an \emph{admissible word} (or just a \emph{word}) of length $|\alpha| = n$
if $A(\alpha_{i - 1}, \alpha_i) = 1$ for $i = 1, \ldots,n - 1$.
Equivalently, there exists $x = x_0 x_1 x_2\cdots\in \X_A$ and $j\in \N$ such that $\alpha = x_{[j,j + n)}$.
If $\alpha = x_{[i,j)}$ and $\beta = y_{[i',j')}$ are words, then $\alpha \beta y_{[j', \infty)}\in \X_A$ if and only if $A(x_{j-1},y_{i'}) = 1$
in which case the concatenation $\alpha \beta$ is again a word.
The topology of $\X_A$ has a basis consisting of compact open sets of the form
\[
    Z_{\alpha} = \{ x\in \X_A\mid x_{[0,|\alpha|)} = \alpha \} \subset \X_A
\]
where $\alpha$ is a word. Let $(\X_A,\sigma_A)$ and $(\X_B,\sigma_B)$ be one-sided shifts of finite type and let $h\colon \X_A\LRA \X_B$ be a homeomorphism. Then $h$ is a \emph{one-sided conjugacy} (or just a \emph{conjugacy}) if $h\circ \sigma_A = \sigma_B\circ h$.

\subsection{Groupoids}
Let $A$ be a finite square $\{0,1\}$-matrix with no zero rows and no zero columns. The \emph{Deaconu-Renault groupoid} \cite{De1995} associated to the one-sided shift of finite type $(\X_A,\sigma_A)$ is
\[
    \G_{A} := \{ (x,n,y)\in \X_A\times \Z\times \X_A\mid \exists k,l\in \N,~n=k-l\colon \sigma_A^k(x) = \sigma_A^l(y) \}
\]
with unit space $\G_A^{(0)} = \{(x,0,x)\in \G_A\mid x\in \X_A\}$. The range map is $r(x,n,y) = (x,0,x)$ and the source map is $s(x,n,y) = (y,0,y)$. The product $(x,n,y)(x',n',y')$ is well-defined if and only if $y=x'$ in which case it equals $(x,n+n',y')$ while inversion is given by ${(x,n,y)}^{-1} = (y,-n,x)$. We can specify a topology on $\G_A$ via a basis consisting of sets of the form
\[
    Z(U,k,l,V) := \{ (x,k-l,y) \in \G_A\mid x\in U, y\in V\},
\]
where $k,l\in \N$ and $U, V\subset \X_A$ are open such that $\sigma_A^k|_U$ and $\sigma_A^l|_V$ are injective and $\sigma_A^k(U) = \sigma_A^l(V)$. If $\alpha,\beta$ are words with the same final letter, we write
\[
    Z(\alpha,\beta) := Z(Z_{\alpha},|\alpha|,|\beta|,Z_{\beta}).
\]
With this topology, $\G_A$ is a second countable, \'etale (i.e., $s, r\colon \G_A\LRA \G_A$ are local homeomorphisms onto $\G_A^{(0)}$) and locally compact Hausdorff groupoid. 
Throughout the paper, we identify the unit space $\G_A^{(0)}$ of $\G_A$ with $\X_A$ via the map $(x,0,x)\mapsto x$. By, e.g.,~\cite[Lemma 3.5]{Sims-Williams}, $\G_A$ is amenable, so the reduced and the full groupoid $C^*$-algebras coincide. We shall refer to the groupoid homomorphism $c_A\colon \G_A\LRA \Z$ given by $c_A(x,n,y) = n$ as the \emph{canonical continuous cocycle}. The pre-image $c_A^{-1}(0) = \{(x,0,y)\in \G_A\mid x,y\in \X_A\}$ is a principal subgroupoid of $\G_A$.

\subsection{Cuntz--Krieger algebras}
Let $A$ be an $N\times N$ matrix with entries in $\{0,1\}$ and no zero rows and no zero columns. The \emph{Cuntz--Krieger algebra} \cite{CK80} $\OO_A$ is the universal unital $C^*$-algebra generated by partial isometries $s_1\ldots,s_N$ subject to the conditions
\begin{align*}
    s_i^*s_j = 0~(i\neq j), \qquad {s_i}^*s_i = \sum_{j=1}^{N} A(i,j) s_j{s_j}^*
\end{align*}
for every $i = 1,\ldots,N$. A word $\alpha = \alpha_1\cdots \alpha_{|\alpha|}$ defines a partial isometry $s_{\alpha} := s_{\alpha_1}\cdots s_{\alpha_{|\alpha|}}$ in $\OO_A$. The \emph{diagonal} subalgebra $\D_A$ is the abelian $C^*$-algebra  generated by the range projections of the partial isometries $s_{\alpha}$ inside $\OO_A$. The algebras $\D_A$ and $C(\X_A)$ are isomorphic via the correspondence $s_{\alpha}{s_{\alpha}}^* \longleftrightarrow \chi_{Z_{\alpha}}$, where $\chi_{Z_{\alpha}}$ is the indicator function on $Z_{\alpha}$. If $A$ is irreducible and not a permutation matrix, then $\OO_A$ is simple and $\D_A$ is maximal abelian in $\OO_A$; in fact, it is a Cartan subalgebra in the sense of~\cite{Renault08}. The \emph{gauge action} $\gamma^A\colon \T\curvearrowright \OO_A$ is determined by $\gamma_z^A(s_i) = zs_i$, for every $z\in \T$ and $i=1,\ldots,N$. The corresponding fixed point algebra is denoted $\F_A$.

The Cuntz--Krieger algebra is a groupoid $C^*$-algebra in the sense that there is a $*$-isomorphism $\OO_A\LRA C^*(\G_A)$ which sends the canonical generators $s_i$ to the indicator functions $\chi_i = \chi_{Z(i,r(i))}$, for $i=1,\ldots,N$ (see, e.g.,~\cite{ggcka}). 
This map takes $\D_A$ to $C(\G_A^{(0)})$ (the latter is identified with $C(\X_A)$) and $\F_A$ to $C^*(c_A^{-1}(0))$. 
The canonical cocycle $c_A\colon \G_A\LRA \Z$ defines an action $\gamma^{c_A}$ on $C_c(\G_A)$ as
\[
    \gamma^{c_A}_z(f)(x, n, y) = z^{c_A(x, n, y)} f(x, n, y) = z^n f(x, n, y),
\]
for $f\in C_c(\G_A)$ and $(x, n, y)\in \G_A$.
In particular, $\gamma^{c_A}_z(s_i) = zs_i$ for $i = 1,\ldots,N$ so $\gamma^{c_A}$ is the gauge action $\gamma^A$ restricted to $C_c(\G_A)$.

Throughout the paper we shall suppress this $*$-isomorphism and simply identify the algebras. The existence of a linear injection $j\colon C^*(\G_A) \LRA C_0(\G_A)$ allows us to think of elements in $\OO_A$ as functions on $\G_A$ vanishing at infinity, cf.~\cite{RenaultPHD} or~\cite{SimsNotes}. We shall do this whenever it be convenient. The inclusion $\iota\colon \G_A^{(0)}\LRA \G_A$ induces a conditional expectation $d_A\colon \OO_A\LRA \D_A$, see \cite[Remark 2.18]{CK80} and \cite[II, Proposition 4.8]{RenaultPHD}. In light of the above, we can think of $d_A(f)$ as the restriction of $f\in \OO_A\subseteq C_0(\G_A)$ to $\X_A$ where we identify $\X_A$ with $\G_A^{(0)}$.

\section{The results}

Let $A$ be an $N\times N$ matrix with entries in $\{0,1\}$ and no zero rows and no zero columns, and let $s_1\ldots,s_N$ be the canonical generators of $\OO_A$. In~\cite{CK80}, J.~Cuntz and W.~Krieger consider a completely positive map $\phi_A\colon \OO_A\LRA \OO_A$ given by
\[
    \phi_A(y) = \sum_{i=1}^N s_i y {s_i}^*,
\]
for $y\in \OO_A$. 
The map $\phi_A$ restricts to a $*$-homomorphism $\D_A\LRA \D_A$. 
Under the identification of $\D_A$ and $C(\X_A)$ we have the relation $\phi_A(f)(x) = f(\sigma_A(x))$ for $f\in \D_A$ and $x\in \X_A$, cf.~\cite[Proposition 2.5]{CK80}.

Put $s := \sum_{i=1}^N s_i$. In this paper, we shall also consider the completely positive map $\tau_A\colon \OO_A\LRA \OO_A$ defined by
\begin{equation}\label{eq:phi_A}
    \tau_A(y) := s y s^* = \sum_{i,j=1}^N s_j y {s_i}^*,
\end{equation}
for $y\in \OO_A$. Note that $\F_A$ is generated by $\bigcup_{k=0}^{\infty} \tau_A^k(\D_A)$ and $\tau_A(\F_A) \subset \F_A$. 
On the level of groupoids, we consider the homomorphism $\epsilon_A\colon \G_A\LRA \G_A$ given by
\[
    \epsilon_A(x,n,y) := (\sigma_A(x), n, \sigma_A(y)),
\]
for $(x,n,y)\in \G_A$. 
Suppose $(x_i, n_i, y_i)\LRA (x, n, y)$ in $\G_A$ as $i\LRA \infty$ and suppose $Z(\mu, \nu)$ is any basic open set containing $(\sigma_A(x), n, \sigma_A(y))$.
Then $n_i = n$ and $\sigma_A(x_i)\in Z_\mu$ and $\sigma_A(\nu)\in Z_\nu$ eventually by continuity of $\sigma_A$.
Hence $(\sigma_A(x_i), n_i, \sigma_A(y_i))\in Z(\mu, \nu)$ eventually, 
so $(\sigma_A(x_i), n_i, \sigma_A(y_i))\LRA (\sigma_A(x), n, \sigma_A(y))$ as $i\LRA \infty$ and $\varepsilon_A$ is continuous.

This induces a map $\epsilon_A^*\colon C_c(\G_A)\LRA C_c(\G_A)$ given by $\epsilon_A^*(f) = f\circ \epsilon_A$, for $f\in C_c(\G_A)$.

\begin{lemma}\label{lem:extend}
    The map $\tau_A\colon \OO_A\LRA \OO_A$ extends the map $\epsilon_A^*$ defined above and we have $d_A\circ \tau_A|_{\D_A} = \phi_A|_{\D_A}$.
\end{lemma}

\begin{proof}
    The generators $s_i$ in $\OO_A$ correspond to the indicator functions $\chi_{i}$ in $C_c(\G_A)$. Inside the convolution algebra we thus have
    \begin{align*}
        \sum_{i,j=1}^N \chi_j\star (f\star {\chi_i}^*)(x,n,y) 
        &= \sum_{j=1}^N \sum_{(z,m,y)\in (\G_A)} \chi_j (x,n-m,z) f(z,m+1,\sigma_A (y)) \\
        &= f(\sigma_A(x),n,\sigma_A(y)),
    \end{align*}
    for $f\in C_c(\G_A)$ and $(x,n,y)\in \G_A$. Here, $\star$ denotes the convolution product in $C_c(\G_A)$. The maps $\epsilon_A^*$ and $\tau_A$ therefore agree on $C_c(\G_A)$. A computation similar to the above shows that $d_A(\tau_A(f)) = \phi_A(f)$, for $f\in \D_A$.
\end{proof}

For the proof of Theorem~\ref{thm}, we need the following lemma which might be of interest on its own.

\begin{lemma}\label{lem:phi-flet}
	Let $A$ and $B$ be finite square $\{0,1\}$-matrices with no zero rows and no zero columns, and let $\Psi\colon \F_A\LRA \F_B$ be a $*$-isomorphism such that $\Psi(\D_A) = \D_B$. Then $\Psi(d_A(f))=d_B(\Psi(f))$ for all $f\in\F_A$. If, in addition, $\Psi\circ \tau_A|_{\F_A} = \tau_B\circ \Psi$, then $\Psi\circ\phi_A|_{\D_A}=\phi_B\circ\Psi|_{{}\D_A}$.
\end{lemma}

\begin{proof}
    The groupoids $c_A^{-1}(0)$ and $c_B^{-1}(0)$ are principal. By~\cite[Proposition 4.13]{Renault08} (see also \cite[Theorem 3.3]{CRST}) and \cite[Proposition 5.7]{Ma2012a} and its proof, there is a groupoid isomorphism $\kappa\colon c_B^{-1}(0)\LRA c_A^{-1}(0)$ and a groupoid homomorphism $\xi\colon c_{A}^{-1}(0)\LRA \T$ such that
    \[
        \Psi(f)(\eta) = \xi(\kappa(\eta)) f(\kappa(\eta)),
    \]
    for $f\in \F_A$ and $\eta\in c_{B}^{-1}(0)$. In particular, $\Psi(f)(x) = f(\kappa(x))$ for $x\in \X_B$ when we identify $x\in \X_B$ with $(x,0,x)\in\G_B^{(0)}$. For $x\in \X_B$, we then see that
    \[
        \Psi(f)(x) = f(\kappa(x)) = f|_{\X_A} (\kappa(x)) = \Psi(f|_{\X_A})(x),
    \]
    that is, $\Psi\circ d_A = d_B\circ \Psi$. If, in addition, $\Psi\circ \tau_A|_{\F_A} = \tau_B\circ \Psi$, then
    \[
        \Psi(\phi_A(f)) = \Psi(d_A(\tau_A(f))) = d_B(\tau_B(\Psi(f))) = \phi_B(\Psi(f)),
    \]
    for $f\in \D_A$.
\end{proof}

We now arrive at our main theorem. If $A$ and $B$ are finite square $\{0,1\}$-matrices with no zero rows and no zero columns, then any isomorphism $\Phi\colon \G_{A}\LRA \G_B$ restricts to a homeomorphism from $\G_A^{(0)}$ to $\G_B^{(0)}$ and thus induces a homeomorphism from $\X_A$ to $\X_B$ via the identification of $\X_A$ with $\G_A^{(0)}$ and the identification of $\X_B$ with $\G_B^{(0)}$. We denote the latter homeomorphism by $\Phi^{(0)}$.

\begin{theorem}\label{thm}
    Let $A$ and $B$ be finite square $\{0,1\}$-matrices with no zero rows and no zero columns, and let $h\colon \X_A\LRA \X_B$ be a homeomorphism. The following are equivalent.
    \begin{enumerate}
        \item[(i)] The homeomorphism $h\colon \X_A\LRA \X_B$ is a conjugacy.
        \item[(ii)] There is a groupoid isomorphism $\Phi\colon \G_{A}\LRA \G_B$ satisfying $c_B\circ\Phi=c_A$, $\Phi^{(0)}=h$, and 
            \begin{equation}\label{eq:groupoid_rel}
                 \Phi\circ \epsilon_A = \epsilon_B\circ \Phi.
            \end{equation}
        \item[(iii)] There is a groupoid isomorphism $\Phi\colon \G_{A}\LRA \G_B$ satisfying $\Phi^{(0)}=h$ and \eqref{eq:groupoid_rel}.
        \item[(iv)] There is a $*$-isomorphism $\Psi\colon \OO_A\LRA \OO_B$ satisfying 
            $\Psi\circ\gamma_z^A=\gamma_z^B\circ\Psi$ for all $z\in \T$, 
            $\Psi\circ d_A=d_B\circ\Psi$, 
            $\Psi(f)=f\circ h^{-1}$ for all $f\in\D_A$, 
            $\Psi\circ\phi_A|_{\D_A}=\phi_B\circ\Psi|_{\D_A}$ and
            \begin{equation}\label{eq:algebra_rel}
                \Psi\circ \tau_A = \tau_B\circ \Psi.
            \end{equation}
        \item[(v)] There is a $*$-isomorphism $\Psi\colon \OO_A\LRA \OO_B$ satisfying $\Psi(\D_A) = \D_B$, 
            $\Psi(f)=f\circ h^{-1}$ for all $f\in\D_A$, and~\eqref{eq:algebra_rel}.
        \item[(vi)] There is a $*$-isomorphism $\Theta\colon\D_A\to\D_B$ satisfying $\Theta(f)=f\circ h^{-1}$ for all $f\in\D_A$, and $\Theta\circ\phi_A|_{\D_A}=\phi_B\circ\Theta$.
    \end{enumerate}   
\end{theorem}

\begin{remark}
As we shall see in the proof, 
if $h\colon \X_A\LRA \X_B$ is a conjugacy, 
then the groupoid isomorphisms $\Phi\colon \G_A\to\G_B$ in (ii) and (iii) can be chosen such that $\Phi((x,n,y)) = (h(x),n,h(y))$ for $(x,n,y)\in\G_A$. 
Also, if $\Phi\colon \G_A\to\G_B$ is a groupoid isomorphism as in (iii) (or (ii)), 
then the $*$-isomorphisms $\Psi\colon \OO_A\LRA \OO_B$ in (iv) and (v) can be chosen to satisfy $\Psi(y)(\eta) = y(\Phi^{-1}(\eta))$ for $y\in\OO_A$ and $\eta\in\G_B$.
\end{remark}

\begin{proof}
    (i) $\implies$ (ii): 
    If $h\colon \X_A\LRA X_B$ is a conjugacy we can define a groupoid homomorphism $\Phi\colon \G_A\LRA \G_B$ by $\Phi((x,n,y)) = (h(x),n,h(y))$ for each $(x,n,y)\in \G_A$. 
    It is clear that $\Phi$ is a bijective algebraic homomorphism.
    In order to see that $\Phi$ is continuous, 
    suppose $(x_i, n_i, y_i) \LRA (x, n, y)$ in $\G_A$ as $i\LRA \infty$ and pick $Z(\mu, \nu)\subset \G_B$ containing $(h(x), n, h(y))$.
    Note that $n_i$  eventually equals $n$.
    As $h$ is continuous, we have $h(x_i)\in Z(\mu)$ and $h(y_i)\in Z(\nu)$ eventually, hence $h(x_i, n, y_i)\in Z(\mu, \nu)$ eventually,
    so $h(x_i, n, y_i)\LRA h(x, n, y)$ as $i\LRA \infty$.
    The argument for $\Phi^{-1}$ is symmetric, so $\Phi$ is a groupoid isomorphism which satisfies $c_B\circ\Phi=c_A$ and $\Phi^{(0)}=h$.
    As $h$ is a conjugacy, $\Phi$ also satisfies~\eqref{eq:groupoid_rel}.

    The implications (ii) $\implies$ (iii) and (iv) $\implies$ (v) are obvious. 

    (iii) $\implies$ (v) and (ii) $\implies$ (iv): 
    A groupoid isomorphism $\Phi\colon \G_A\LRA \G_B$ with $\Phi^{(0)} = h$ induces a $*$-isomorphism $\Psi\colon \OO_A\LRA \OO_B$ with $\Psi\circ d_A = d_B\circ \Psi$. 
    In particular, $\Psi(\D_A) = \D_B$ and $\Psi(f) = f\circ h^{-1}$, for $f\in \D_A$. Since $\Phi$ satisfies~\eqref{eq:groupoid_rel}, we also have $\Psi\circ \tau_A = \tau_B\circ \Psi$. This is (v). 
    If, in addition, $\Phi$ satisfies $c_B \circ \Phi = c_A$, then $\Psi(s_i) = 1_{\Phi(Z(i, r(i)))}\in c_B^{-1}(\{0\})$ so
    \[
        \Psi(\gamma^A_z(s_i)) = z \Psi(1_{\Phi(Z(i, r(i)))}) = \gamma^B_z(\Psi(s_i)),
    \]
    for $i=1,\ldots,N$. It follows that $\Psi\circ \gamma_z^A = \gamma_z^B\circ \Psi$, for every $z\in \T$. 
    In particular, this implies that $\Psi(\F_A) = \F_B$. By Lemma~\ref{lem:phi-flet}, it follows that $\Psi\circ \phi_A|_{\D_A} = \phi_B\circ \Psi|_{\D_A}$. This is (iv).
        
    (v) $\implies$ (vi): If $\Psi$ preserves the diagonal and satisfies~\eqref{eq:algebra_rel}, then
    \[
        \Psi\big(\bigcup_{k = 0}^\infty \tau_A^k (\D_A) \big) = \bigcup_{k = 0}^{\infty} \tau_B^k(\D_B). 
    \]
    As $\F_A$ is generated by $\bigcup_{k = 0}^{\infty} \tau_A^k(\D_A)$ as a $\mathrm C^*$-algebra, it follows that $\Psi(\F_A) = \F_B$.
    Lemma~\ref{lem:phi-flet} then entails that $\Psi\circ \phi_A|_{\D_A} = \phi_B\circ \Psi|_{\D_A}$.

    (vi) $\implies$ (i): The relation $\Theta\circ \phi_A|_{\D_A} = \phi_B\circ \Theta$ and the fact that $\phi_A(f)(x) = f(\sigma_A(x))$ for $f\in \D_A$ and $x\in \X_A$
    ensures that $h$ is a conjugacy by Gelfand duality.
\end{proof}

\begin{corollary}
    Let $A$ and $B$ be finite square $\{0,1\}$-matrices with no zero rows and no zero columns. The following are equivalent.
    \begin{enumerate}
        \item[(i)] The one-sided shifts $(\X_A,\sigma_A)$ and $(\X_B,\sigma_B)$ are one-sided conjugate.
        \item[(ii)] There is a groupoid isomorphism $\Phi\colon \G_A\LRA \G_B$ satisfying $\Phi\circ \epsilon_A = \epsilon_B\circ \Phi$.
        \item[(iii)] There is a $*$-isomorphism $\Psi\colon \OO_A\LRA \OO_B$ satisfying $\Psi(\D_A) = \D_B$ and $\Psi\circ \tau_A = \tau_B\circ \Psi$.
    \end{enumerate}
\end{corollary}

One-sided shifts $(\X_A,\sigma_A)$ and $(\X_B,\sigma_B)$ are \emph{eventually conjugate} if there exist a homeomorphism $h\colon \X_A\LRA \X_B$ and $L\in \N$ such that 
\[
    \sigma_B^L(h(\sigma_A(x)) = \sigma_B^{L+1}(h(x)), \qquad 
    \sigma_A^L(h^{-1}(\sigma_B(y)) = \sigma_A^{L+1}(h^{-1}(y)),
\]
for $x\in \X_A$ and $y\in \X_B$. 
The above theorem should be compared to~\cite[Theorem 3.5]{Mat17cts} and~\cite[Corollary 4.2]{CR} which characterize one-sided eventual conjugacy. 

\begin{example}
Consider the following two graphs.

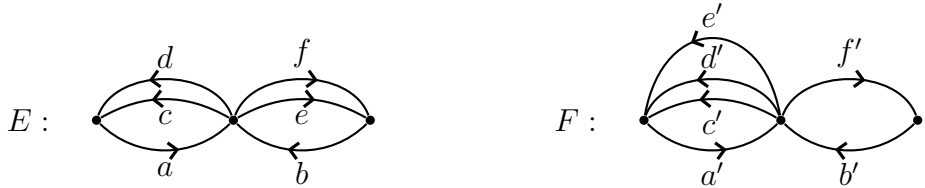
\begin{center}
\begin{tikzpicture}[scale=1.8, ->-/.style={thick, decoration={markings, mark=at position 0.6 with {\arrow{Straight Barb[line width=0pt 1.5]}}},postaction={decorate}}, vertex/.style={inner sep=0pt, circle, fill=black}]
\node at (-0.5,0) {$E:$};
\node at (3.5,0) {$F:$};
\node (E1) at (0,0) [vertex] {.};
\node (E2) at (1,0) [vertex] {.};
\node (E3) at (2,0) [vertex] {.};
\node (F1) at (4,0) [vertex] {.};
\node (F2) at (5,0) [vertex] {.};
\node (F3) at (6,0) [vertex] {.};
\draw[->-, bend right=45] (E1) to node[below] {$a$} (E2);
\draw[->-, bend right=30] (E2) to node[below] {$c$} (E1);
\draw[->-, bend right=70] (E2) to node[above] {$d$} (E1);
\draw[->-, bend left=45] (E3) to node[below] {$b$} (E2);
\draw[->-, bend left=30] (E2) to node[below] {$e$} (E3);
\draw[->-, bend left=70] (E2) to node[above] {$f$} (E3);
\draw[->-, bend right=45] (F1) to node[below] {$a'$} (F2);
\draw[->-, bend right=30] (F2) to node[below] {$c'$} (F1);
\draw[->-, bend right=70] (F2) to node[above] {$d'$} (F1);
\draw[->-, bend left=45] (F3) to node[below] {$b'$} (F2);
\draw[->-, bend right=80, looseness=2] (F2) to node[above] {$e'$} (F1);
\draw[->-, bend left=70] (F2) to node[above] {$f'$} (F3);
\end{tikzpicture}
\end{center}

Let $\X_E$ be the one-sided edge shift of $E$ and let $\X_F$ be the one-sided edge shift of $F$. Define a homeomorphism $h\colon \X_E\LRA \X_F$ by sending ${(x_n)}_{n\geq 0}\in \X_E$ to ${(y_n)}_{n\geq 0}\in \X_F$ where
\[
    y_n = 
    \begin{cases}
        a', & \textrm{if $n>0$ and $x_{n-1}=e$}, \\
        (x_n)' & \textrm{otherwise}.
    \end{cases}
\]
E.g., $h(e b c a f\ldots) = e' a' c' a' f'\ldots$ while $h(\sigma_E(e b c a f\ldots)) = h(b c a f\ldots) = b' c' a' f'\ldots$
and $h^{-1}(a' e' a' f' \ldots) = a e b f \ldots$.
Observe that
\[
    \sigma_{F}^2(h(x)) = \sigma_{F}(h(\sigma_{E}(x))), \qquad \sigma_{E}^2(h^{-1}(y)) = \sigma_{E}(h^{-1}(\sigma_{F}(y))),
\]
for $x\in \X_E$ and $y\in \X_F$. Hence $\X_E$ and $\X_F$ are eventually conjugate.
On the other hand, the total amalgamation of $E$ is the graph
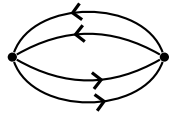
\begin{center}
\begin{tikzpicture}[scale=2, ->-/.style={thick, decoration={markings, mark=at position 0.6 with {\arrow{Straight Barb[line width=0pt 1.5]}}},postaction={decorate}}, vertex/.style={inner sep=0pt, circle, fill=black}]
\node (G1) at (0,0) [vertex] {.};
\node (G2) at (1,0) [vertex] {.};
\draw[->-, bend right=30] (G1) to (G2);
\draw[->-, bend right=70] (G1) to (G2);
\draw[->-, bend right=30] (G2) to (G1);
\draw[->-, bend right=70] (G2) to (G1);
\end{tikzpicture}
\end{center}
while the total amalgamation of $F$ is $F$ itself. It thus follows that $\X_E$ and $\X_F$ are \emph{not} conjugate, cf.\ e.g.,~\cite[Section 13.8]{LM} or~\cite[Theorem 2.1.10]{Kitchens}.
\end{example}

\end{document}